\documentclass[a4paper,12pt]{article}
\usepackage{amsfonts}
\usepackage{latexsym}
\usepackage{amsthm}
\usepackage{amssymb}
\usepackage{amsmath}
\usepackage{eufrak}
\usepackage{mathrsfs}
\usepackage{geometry}
\usepackage{graphics}
\usepackage{comment}
\usepackage[all]{xy}

\newtheorem*{pps}{Proposition \ref{nt1}}
\newtheorem*{tma}{Theorem \ref{nt3}}

\newtheorem{thm}{Theorem}[section]
\newtheorem{lema}[thm]{Lemma}
\newtheorem{defi}[thm]{Definition}
\newtheorem{exe}[thm]{Example}
\newtheorem{prop}[thm]{Proposition}
\newtheorem{cor}[thm]{Corollary}

\newcommand{\lra}{ \longrightarrow }
\newcommand{\Ker}{{\rm Ker}}
\renewcommand{\Im}{{\rm Im}}

\author{Natalia A. Viana Bedoya\thanks{The beginning of this   work appears in my PhD thesis 
    written under supervision of the Professor Daciberg Lima
    Gon{\c{c}}alves and supported by FAPESP 
    process 03/12309-4.}  \ and Daciberg Lima Gon\c calves}

\title{Decomposability problem on branched coverings}
\date{}
\begin{document} 

\maketitle
\begin{abstract}
 Given a
branched covering of degree $d$ between closed 
surfaces, it determines  
a collection of partitions of $d$, \emph{the branch data}. In 
this work we show that any branch data are realized by an indecomposable
primitive branched covering  on a
connected closed surface $N$ with $\chi(N) \leq 0$. This shows that decomposable and indecomposable realizations may coexist. Moreover, we characterize
the branch data of a decomposable primitive branched
covering. 
\\[0.2cm]
\\[0.2cm]
\end{abstract}
\section*{Introduction}
We begin by describing   some historical facts and results related with the problem considered in this work.   In 1957  Borsuk and Molski \cite{BM} asked about the existence of a continuous map of finite
order\footnote{A continuous map $\phi$ defined on a space $X$ is said to be of \emph{order} $\leq k \in
\mathbb{Z}^{+}$ if for any $y \in \phi(X)$, $\phi^{-1}(y)$ contains at most
$k$ points.}, which is not a composition of 
\emph{simple maps} (maps of order $\leq 2$). In 1959, Sieklucki \cite{S} showed that every such a map defined on a compact finite
dimensional metric space is a composition of simple maps, and  gave an example on an infinite dimensional compact space which cannot be
decomposed.  
In 1973, Baildon \cite{B} showed that if an open
surjective map of finite order  between closed surfaces  is a composition of
$k$ simple open surjective maps, then its order is equal to $2^{k}$. In 2002, Krzempek  \cite{KJ} constructed covering maps on locally arcwise connected continua that are not factorizable into covering maps of order $\leq n-1$, for all $n$.  In 2002,
 Bogataya, Bogaty{\u\i} and Zieschang \cite{BBZ} extended Baildon's theorem to
compositions of arbitrary open maps and showed that the order of a product
(composition)  is the product of the orders. Moreover, they gave an
example  of a 4-fold covering of a surface of genus 2 by a surface of genus 5
that cannot be represented as a composition of two non-trivial open maps.  

In \cite{Why}, Whyburn showed that finite order open maps on closed surfaces are branched coverings. 
 The
purpose of this work is to  answer the question  whether or not a primitive (surjective on $\pi_{1}$)
branched 
covering of degree $d \in \mathbb{Z^{+}}$ between closed connected surfaces  is
decomposable by non-trivial coverings of degree $< d$. We will impose the condition of
surjection on the fundamental group because  a
non-primitive branched covering is always decomposable (see \cite{BGZ}).

A branched covering $\phi: M \lra N$ of degree $d$ between closed connected surfaces determines  
a finite collection $\mathscr{D}$  of partitions  of $d$, \emph{the branch
 data}. Conversely, given $\mathscr{D}$ and $N$,  Husemoller in
 \cite{Hu} and Ezell in \cite{Ez} gave a necessary and sufficient condition ({\it Hurwitz's condition for $\mathscr{D}$}, see Section \ref{pbc}) for the existence of a branched covering $\phi: M \lra N$ between connected closed surfaces with $\mathscr{D}$
 as branch data, whenever  $\chi(N)\leq 0$. In this case we say that {\it $\mathscr{D}$ is realizable by $\phi$ on $N$}. Moreover Bogatyi, Gon{\c{c}}alves,
 Kudryavtseva and Zieschang, in \cite{BGKZ2} and \cite{BGKZ1}, showed under that
 condition,  that the branched covering can be chosen primitive. 
 
 A collection of
 partitions of $d$ satisfying Hurwitz's 
 condition  will be called \emph{admissible}. The main result is:
 \begin{tma}
Every non-trivial admissible data are  realized on any  $N$ with $\chi(N)\leq 0$,  by an
indecomposable primitive branched covering. 
\end{tma}
We also
  characterize  admissible data realized by  a 
 decomposable primitive branched covering over $N$, {\it decomposable data on $N$},
 by defining a special factorization on it (see Section \ref{decomp}).
\begin{pps}
Admissible data $\mathscr{D}$  are decomposable on $N$,  with  $\chi(N)\leq 0$,
 if and only if there exists a factorization of $\mathscr{D}$
  such that its first factor is  non-trivial  admissible data.
\end{pps}

The problem we solve here provides a contribution for the understanding and  for a possible classification of  branched coverings. After the  realization results provided by
\cite{Hu} and  \cite{EKS}, a substantial contribution was obtained in \cite{BGKZ2}, \cite{BGKZ1}, \cite{BGZ}, \cite{GKZ} by solving the realization problem under the hypothesis 
that the covering is primitive, i.e. the induced map on the fundamental group is surjective. Now we further explore this realization  type of result  studying the 
decomposability by possibly indecomposable branched coverings. 
Also our problem is related with the Inverse Galois problem (see  for example the references \cite{Muller} and  \cite{GN}) and with a construction of primitive and imprimitive monodromy groups as treated in \cite{LZ}. Besides the facts mentioned  above,  this problem seems interesting in its own right.

The paper is divided into  four sections. In Section 1, we quote the main
definitions and some results related to branched coverings. In Section 2, we characterize the 
branch data of a  decomposable primitive 
branched covering. In Section 3, we assert that if $N$ is either the torus or the  Klein bottle, an admissible
partition is realized on $N$ by an indecomposable primitive branched covering. Then we generalize it for every admissible data and any  $N$ with $\chi(N)\leq 0$.
In Section 4 we prove the assertion in Section 3.

{\bf Acknowledgements} The second author would like to express his gratitude to Elena Kudryavtseva and Semeon Bogaty{\u\i} from Lomonosov Moscow State University, for many helpful 
and fruitful conversations about these problems    during his visit to the  Chair  of  Differential Geometry and Applications-MSU,  May 2002.  He also would like to thank the 
warm hospitality of the Chair of  Differential Geometry and Applications  during his visit. 
\section{Preliminaries, terminology and notation}
  \subsection{On permutation groups}
We denote by $\Sigma_{d}$ the symmetric group on a set $\Omega$ with $d$ elements and
 by $1_{d}$ its identity element. If $\alpha \in \Sigma_{d}$ and $x \in
 \Omega$, $x^{\alpha}$ is the image of $x$ by $\alpha$. An explicit permutation $\alpha$ will be written either as a product of disjoint cycles, i.e. its \emph{cyclic decomposition},  or in the following way:
 \begin{displaymath}
 \mathbf{\alpha=} 
\left( \begin{array}{ccccccc}
1 & 2& \dots & 2k+1\\ 
1^{\alpha} &2^{\alpha}& \dots& (2k+1)^{\alpha}
\end{array} \right),
\end{displaymath}   
it depends on our convenience. The set of lengths of the
 cycles in the cyclic decomposition of $\alpha$, including the trivial ones, defines a partition of $d$, say
 $D_{\alpha}=[d_{\alpha_{1}},\dots,d_{\alpha_{t}}]$,  called 
 \emph{the cyclic structure of} $\alpha$. Define
 $\nu(\alpha):=\sum_{i=1}^{t}(d_{\alpha_{i}}-1)$, then $\alpha$ will be an \emph{even permutation} if $\nu(\alpha) \equiv 0 \pmod{2}$.
Given  a partition $D$ of $d$, we say  $\alpha \in
D$ if the cyclic structure of $\alpha$ is
$D$ and  we put $\nu(D):=\nu(\alpha)$.

For $1<r \leq d$, a permutation $\alpha \in \Sigma_{d}$ is called a
$r$-\emph{cycle} if in its cyclic decomposition its unique non-trivial cycle
has 
length $r$. Permutations $\alpha,\beta \in \Sigma_{d}$ are \emph{conjugate}
if there is $\lambda \in \Sigma_{d}$  such that
$\alpha^{\lambda}:=\lambda \alpha \lambda^{-1}=\beta$. It is a known fact that
conjugate permutations have the same cyclic structure.  

  Given a permutation group $G$ on $\Omega$  and $x \in \Omega$,  one  defines the {\it isotropy subgroup of $x$},
 $G_{x}:=\{g \in G: x^{g}=x\}$ , and the {\it orbit of $x$ by $G$}, $x^{G}:=\{x^{g}:g\in G\}$.

 For $H \subset G$,  the subsets $Supp(H):=\{x
 \in \Omega: x^{h} \neq x \textrm{ for some $h \in H$}\}$ and
 $Fix(H):=\{x \in \Omega:x^{h}=x \textrm{ for all $h \in H$}\}$ are defined.
 For $\Lambda \subset \Omega$ and $g \in G$,
 $\Lambda^{g}:=\{y^{g}:y \in \Lambda\}$.
 
  $G$ is said to be
 \emph{transitive} if for all
 $x,y \in \Omega$ there is $g \in G$ such that
 $x^{g}= y$.
 A nonempty subset $\Lambda \subset \Omega$ is a \emph{block} of
 a transitive $G$ if
 for each $g \in G$ either $\Lambda^{g} = \Lambda$ or
 $\Lambda^{g} \cap \Lambda =\emptyset$. A block $\Lambda$ is
 \emph{trivial} if either $\Lambda =\Omega$ or $\Lambda=\{ x \}$ for some $x \in
 \Omega$. Given a block $\Lambda$ of $G$, the set
 $\Gamma:=\{\Lambda^{\alpha}:\alpha \in G\}$ defines a partition of $\Omega$
 in blocks. This set is called \emph{a system of blocks
 containing} $\Lambda$ and the cardinality of $\Lambda$ divides the cardinality of $\Omega$. $G$ acts naturally on $\Gamma$.
 A transitive permutation group is
 \emph{primitive} 
 if it determines only  trivial blocks. Otherwise it is
 \emph{imprimitive}.
\begin{exe}\label{e1}
A transitive permutation group $G<\Sigma_{d}$ containing a  $(d-1)$-cycle is 
  primitive.  For, without loss of generality let us suppose that
 $g=(1\dots d-1)(d) \in G$. Then   
 any  proper subset $\Lambda$ of $\{1,\dots,d\}$ containing $d$ and at least one more element satisfies $\Lambda^{g}\neq \Lambda$ and
$\Lambda^{g}\cap \Lambda \neq \emptyset$. Thus the blocks
of  $G$ are trivial and $G$ is primitive.
\end{exe}

\begin{exe}\label{REF}
 If $gcd(\ell, d)=1$ and $\ell$ is greater than any non-trivial divisor of $d$ then any transitive permutation group $G < \Sigma_d$
 containing an $\ell$-cycle is primitive $($this holds, for example, if $d=2\ell\pm1$  $)$. In fact, we can assume that $G$ contains the cycle $(1,\dots,\ell)$. If there is a block of $G$ containing two elements $i$ and $j$ with $i \leq \ell$ and $j>\ell$ then it also contains $1,\dots,\ell$, thus the cardinality of the block is $\geq\ell+1$. Hence it equals $d$ and the block is trivial. Otherwise the cardinality of each block divides both $\ell$ and $d-\ell$, hence it equals 1, thus all blocks of $G$ are trivial. Hence $G$ is a primitive permutation group.
\end{exe}
\begin{prop}[\cite{DM}, Cor. 1.5A]\label{dixon}
Let $G$ be a transitive permutation group on a set $\Omega$ with at least two
points. Then $G$ is primitive if and only if each isotropy subgroup $G_{x}$,
for $x \in \Omega$,
is a maximal subgroup of $G$. \qed
\end{prop}
\subsection{On branched coverings between closed surfaces}\label{pbc}
A surjective continuous open map $\phi:M \lra N$ between closed surfaces such 
that:
\begin{itemize}
\item for $x \in N$, $\phi^{-1}(x)$ is a totally disconnected set, and
\item there is a non-empty discrete set $B_{\phi} \subset N$  such that the 
restriction \break $\hat{\phi}:=\phi|_{M-\phi^{-1}(B_{\phi})}$ is an ordinary
unbranched 
covering of degree $d$,
\end{itemize}
is called a \emph{branched covering of degree d over N}
 and it is  denoted  by
$(M,\phi,N,B_{\phi},d)$.  $N$ is \emph{the base surface}, $M$ is
\emph{the covering surface} and $B_{\phi}$ is \emph{the branch point set}. Its \emph{associated unbranched covering} is denoted by 
$(\widehat{M},\hat{\phi},\widehat{N},d)$, where
$\widehat{N}:=N-B_{\phi}$ and $\widehat{M}:=M-\phi^{-1}(B_{\phi})$.
The set $B_{\phi}$ is  just
the image of the points in $M$ in which $\phi$ fails to be a local
homeomorphism, then each $x \in B_{\phi}$ determines a non-trivial
partition $D_{x}$ of $d$, defined by the local degrees of $\phi$ on each component
in the preimage of a small disk $U_{x}$ around  $x$, with
$U_{x}\cap B_{\phi}=\{x\}$. The collection  $\mathscr{D}:=\{D_{x}\}_{x\in B_{\phi}}$ is called
the  \emph{branch data } $\mathscr{D}$ and its  \emph{total defect} is the positive integer
defined by 
$\nu({\mathscr{D}}):=\sum_{x \in B_{\phi}} \nu(D_{x})$. The total defect  satisfies
the \emph{Riemann-Hurwitz formula} (see \cite{EKS}): 
\begin{eqnarray}\label{rhf}
\nu(\mathscr{D})=d \chi(N)-\chi(M).
\end{eqnarray}

Associated to $(M,\phi,N,B_{\phi},d)$ we have a permutation group, {\it the monodromy group of $\phi$}, given by the
image of the \emph{Hurwitz's representation}
\begin{eqnarray}\label{Hr}
 \rho_{\phi}: \pi_{1}(N-B_{\phi},z) \lra \Sigma_{d}
 \end{eqnarray}
that sends each class $\alpha \in \pi_{1}(N-B_{\phi},z)$ to the permutation of $\phi^{-1}(z)=\{z_{1},
\dots,z_{d}\}$, 
which indicates the terminal point of the lifting  of  a loop
in $\alpha$ after fixing  the
initial point. In particular,  for $x \in B_{\phi}$, let $c_{x}$
be a path
 from $z$ to a small circle $a_{x}$ about $x$ and define the loop class 
 $\mathbf{u}_{x}:=[c_{x}a_{x}c_{x}^{-1}]$. Then the cyclic structure of the permutation 
 $\alpha_{x}:=\rho_{\phi}(\mathbf{u}_{x})$ is given by $D_{x}$ and 
$\nu(\prod_{x \in B_{\phi}} \alpha_{x}) \equiv \nu(\mathscr{D}) \pmod{2}$.

In the sequel,  $N$ will denote a connected closed surface with $\chi(N)\leq 0$. Then $N$ is either
 the connected sum of $g\geq 1$ tori, $N=T_{g}$, or  the connected
 sum  of 
 $g\geq 2$ projective planes, $N=P_{g}$. If $B_{\phi}=\{x_{1},\dots,x_{t}\}$, we
 adopt the following presentations for the respective
 fundamental groups:
\begin{displaymath} \label{gf1}
\begin{small}
\pi_{1}(N-B_{\phi},z)= \left\{ \begin{array}{ll}
\langle \mathbf{u}_{x_{1}},... ,\mathbf{u}_{x_{t}},
a_{1},b_{1},... ,a_{g},b_{g} |
\mathbf{u}_{x_{1}}... \mathbf{u}_{x_{t}}[a_{1},b_{1}]
... [a_{g},b_{g}]=1 \rangle &, \textrm{$N=T_{g}$} \\
\langle \mathbf{u}_{x_{1}},... ,\mathbf{u}_{x_{t}}, a_{1}, ...,
a_{g} | \mathbf{u}_{x_{1}} ... \mathbf{u}_{x_{t}}a_{1}^{2} ...
a_{g}^{2}=1
  \rangle &, \textrm{$N= P_{g}$} \\
\end{array} \right.
\end{small}
\end{displaymath}
 In the special case $N=P_{2}$ and
 $B_{\phi}=\{x\}$, we work also
 with the presentation $\langle \mathbf{u}_{x}, a_{1},
a_{2} | \mathbf{u}_{x}a_{1}a_{2}a_{1}a_{2}^{-1}=1 \rangle$ and this
 will be clear in the context.
Note that 
 $\rho_{\phi}(\prod_{i=1}^{t}\mathbf{u}_{x_{i}})$ is always an even permutation. This
 necessary condition is known as \emph{Hurwitz's condition} and it is
 equivalent to:
\begin{eqnarray}\label{hc2}
 \nu(\mathscr{D}) \equiv 0
 \pmod{2}.
\end{eqnarray}
 
 \begin{thm}[See \cite{Hurwitz}]
 Given $N$ and a finite collection $\mathscr{D}$ of partitions  of $d$, if it is possible to define a representation $\pi_{1}(N-F,z) \lra \Sigma_{d}$ like $\rho_{\phi}$ such that its image is a transitive permutation group, where 
$F\subset N$ is a finite set with the same cardinality as $\mathscr{D}$, then $\mathscr{D}$ is realizable on $N$.\qed
\end{thm}
\begin{thm}[In \cite{Hu} the orientable case. In
\cite{Ez} the nonorientable case.] 
If $\mathscr{D}$ is admissible  then $\mathscr{D}$ is 
realizable on any $N$ with $\chi(N) \leq0$.\qed
\end{thm}

\begin{thm}[In \cite{BGKZ1} geometrically. In  \cite{BGKZ2} algebraically]\label{D}
Any non-trivial admissible data $\mathscr{D}$ are realizable 
by a primitive branched covering on any  $N$ with $\chi(N)\leq 0$.\qed
\end{thm}

Given a covering, it is
\emph{decomposable} if it can be written as a composition of two non-trivial
coverings (i.e., both with degree bigger than 1), otherwise it is called
\emph{indecomposable}. In a decomposition of a branched
covering which is not a covering at least one of its components is a branched covering which is not a covering. Moreover,
since the degree of a decomposable covering is the product of the degrees of
its components (see \cite{BBZ}, theorem 2.3), we are interested in branched
coverings with a non-prime degree.

In order to simplify  notation, given $(M,\phi,N,B_{\phi},d)$, we make the
 identifications:  $\rho:=\rho_{\phi}$, $G:=\Im(\rho)$, and, for a fixed $z \in \widehat{N}$,
recall that $G$ is a permutation group on $\phi^{-1}(z)=\{z_{1},
\dots,z_{d}\}$, and $G_{z_{i}}$ is the isotropy subgroup of $z_{i}$.
\begin{lema}\label{isotropy}
Let $(M,\phi,N,B_{\phi},d)$ be a branched covering. Then in the sequence
$\xymatrix{
\pi_{1}(\widehat{M},z_{i}) \ar[r]^{\hat{\phi}_{\#}}& \pi_{1}(\widehat{N},z)
\ar[r]^{\rho}& G}$
we have $\rho^{-1}(G_{z_{i}})=\hat\phi_{\#}(\pi_{1}(\widehat{M},z_{i}))$.
\end{lema}
\begin{proof}
Consider $\alpha \in
\rho^{-1}(G_{z_{i}})$, then $z_{i}^{\rho(\alpha)}=z_{i}$ and the
lifting $\tilde{\alpha}$ of $\alpha$ with initial point $z_{i}$ is a class
in 
$\pi_{1}(\widehat{M},z_{i})$, and $\hat{\phi}_{\#}(\tilde{\alpha})=
\alpha$. Conversely, if $\alpha \in
\hat{\phi}_{\#}(\pi_{1}(\widehat{M},z_{i}))$, its 
lifting with initial point $z_{i}$ is an element of
$\pi_{1}(\widehat{M},z_{i})$, then $\rho(\alpha) \in G_{z_{i}}$. 
\end{proof}

Now we establish a version of Ritt's Theorem (see \cite{R}) for  branched coverings.
\begin{prop}\label{yo}
A  branched covering
 is decomposable if and only if its monodromy
group is imprimitive.
\end{prop} 
\begin{proof}
Let $(M,\phi,N,B_{\phi},d)$ be decomposable. Then there
is a surface $K$ and branched coverings $\psi,\eta$ of degrees $w,u$ respectively
such 
that $\phi=\eta \psi$ and $d=uw$ ($u,w>1$). 
Define $\widehat{K}:=K-\eta^{-1}(B_{\phi})$ and let $\hat{\phi}$,
$\hat{\psi}$, $\hat{\eta}$ be the restrictions of $\phi$ and $\psi$ on
$\widehat{M}$ and $\eta$ 
on $\widehat{K}$. For a fixed $z \in  \widehat{N}$, let us consider $z_{1} \in
\hat{\phi}^{-1}(z)$ and $y_{1}:=\hat{\psi}(z_{1})$. We have the
commutative diagrams:  
\begin{footnotesize}
$$$$
\xymatrix{
M \ar[d]_{\phi} \ar[dr]^{\psi}& \\
N  & K \ar[l]_{\eta} } 
\qquad  \qquad
\xymatrix{
(\widehat{M},z_{1}) \ar[d]_{\hat{\phi}} \ar[dr]^{\hat{\psi}}& \\
(\widehat{N},z)  & (\widehat{K},y_{1}) \ar[l]_{\hat{\eta}}} 
\qquad \qquad
\xymatrix{
\pi_{1}(\widehat{M},z_{1}) \ar[d]_{\hat{\phi}_{\#}} \ar[dr]^{\hat{\psi}_{\#}}& \\
\pi_{1}(\widehat{N},z)  & \pi_{1}(\widehat{K},y_{1})
\ar[l]_{\hat{\eta}_{\#}}}
$$$$
\end{footnotesize}
that are equivalents to the following sequence of proper subgroups:
\begin{eqnarray}\label{comm}
\hat{\phi}_{\#}(\pi_{1}(\widehat{M},z_{1})) \lneqq
\hat{\eta}_{\#}(\pi_{1}(\widehat{K},y_{1})) \lneqq \pi_{1}(\widehat{N},z)
\end{eqnarray} 
with index
  $|\hat{\eta}_{\#}(\pi_{1}(\widehat{K},y_{1})):\hat{\phi}_{\#}(\pi_{1}(\widehat{M},z_{1}))|=w$ 
  and 
  $|\pi_{1}(\widehat{N},z):\hat{\eta}_{\#}(\pi_{1}(\widehat{K},y_{1}))|=u$. 
By applying $\rho$ on (\ref{comm}) and by  using Lemma
  \ref{isotropy}, we obtain proper subgroups  
$G_{z_{1}} \lneqq \rho(\hat{\eta}_{\#}(\pi_{1}(\widehat{K},y_{1}))) \lneqq G$ with the same index as 
above, since $ ker \rho < \rho^{-1}(G_{z_1})=\widehat{\phi}_{\#}(\pi_1(\widehat{M},z_1))$. Then
  $G_{z_1}$ is not a maximal subgroup and, by
 Proposition \ref{dixon}, $G$ is imprimitive.

Conversely, if $\rho:\pi_{1}(\widehat{N},z) \lra
G$ is Hurwitz's representation of  $(M,\phi,N,B_{\phi},d)$ and $G$  
is imprimitive, by Proposition 1.2 there exists 
a proper subgroup $H_{1}<G$ such that 
 $\rho^{-1}(G_{z_{1}}) \lneqq \rho^{-1}(H_{1}) \lneqq \pi_{1}(\widehat{N},z)$.
 Let $u>1$ be the
index of $\rho^{-1}(H_{1})$ in $\pi_{1}(\widehat{N},z)$ and let
$(\widehat{K},\hat{\eta}, \widehat{N},u)$ be the unbranched covering
determined by 
$\rho^{-1}(H_{1})$. Since
$\rho^{-1}(G_{z_{1}})=\hat \phi_{\#}(\pi_{1}(\widehat{M},z_{1}))$ (see Lemma 
\ref{isotropy}),  there is a lifting $\hat{\psi}:\widehat{M}\lra
\widehat{K}$ of $\hat{\phi}$ and we have
a commutative diagram of unbranched coverings.
 We want to
extend it to $M, N$ and a compactification $K$ of $\widehat{K}$. This
extension is
possible by applying cover space theory to the diagram's restriction on
small 
circles about the elements of $B_{\phi}$.
\end{proof}
A map is \emph{orientation-true} if it maps orientation
preserving loops to orientation preserving loops and orientation reversing
loops to orientation reversing loops. 
A  branched covering  is an orientation-true map (see \cite{GKZ}).

\bigskip
\noindent {\bf Remark.} The torsion part of the abelianized $\pi_1(P_g)$ with $g\geq 1$ is a cyclic subgroup of order $2$. In the presentation
 $\pi_1(P_g)=\langle a_1,...,a_g|a_1^2....a_g^2=1\rangle$ this subgroup of $\pi_1(P_g)_{ab}$ is generated by the unique element $a_1+...+a_g$ of order $2$.

\begin{prop}\label{indice}
Let $(K,\eta,N,B_{\eta},u)$ be a primitive branched covering. Then for all $w \in \mathbb{N}$ there exists a
subgroup 
$H$ in $\pi_{1}(K)$ of index $w$ such that $\eta_{\#}|_{H}$ is an epimorphism. 
\end{prop}
\begin{proof}
By the Correspondence theorem (see \cite{J}, chapter 1, theorem 1.8)  and since $\eta_{\#}$ is an epimorphism, it is enough to work with  the abelianized groups.
If $N=T_{g}$ then by (\ref{rhf}), $K=T_{h}$, with $h>g>0$. Let
$\bar{\eta}_{\#}:\pi_{1}(T_{h})_{ab} \lra \pi_{1}(T_{g})_{ab}$ 
be the induced by $\eta_{\#}$ epimorhism of the 
abelianized fundamental groups.  Note that
$\pi_{1}(T_{h})_{ab} \cong \langle
a_{1},\dots,a_{2(h-g)},b_{1},\dots,b_{2g}\rangle$, 
  where the $a_{i}$'s are generators of $\Ker(\bar{\eta_{\#}})$ 
and the  $\bar \eta_{\#}(b_{j})$'s are  generators
 of $\pi_{1}(T_{g})_{ab}$. 
  Define the subgroup $\bar{H}=\langle
  wa_{1},a_2,\dots,a_{2(h-g)},b_{1},\dots,b_{2g}\rangle$ of
  index  $w$. The restriction $\bar{\eta_{\#}}|_{\bar{H}}$ 
  is an  epimorphism.

If $N=P_{g}$ with $g \geq 2$, since $\eta$ is primitive
and orientation-true,
necessarily $K=P_{h}$ with $h\geq
2g$ by (\ref{rhf}). Since $\pi_1(K)_{ab}$  and $\pi_1(N)_{ab}$ have isomorphic torsion parts  (both $\mathbb{Z}_2$, see Remark above),
the restriction  of the epimorphism $\bar \eta_{\#}$ to the torsion   parts is an  isomorphism. By quoting out $\pi_1(K)_{ab}$ and $\pi_1(N)_{ab}$  by their
 torsion parts, one obtains an epimorphism between free Abelian groups of ranks $h-1$ and $g-1$, hence $\Ker (\bar\eta_{\#})$ is a free Abelian group of rank $h-g$.  Thus,  
 we can choose a base $\{
a_{1},\dots,a_{h}\}$ of $\pi_{1}(P_{h})_{ab}$ 
such that $\Ker (\bar{\eta}_{\#})=\langle  a_{1},\dots,a_{h-g} \rangle$, a free Abelian 
group of rank $h-g$.
In particular, $a_1$  has infinite order. The restriction of $\bar \eta_{\#}$ to
the subgroup  $\bar{H}=\langle
wa_{1},a_{2},\dots,a_{h}\rangle$, of index $w$, is an
epimorphism. 
\end{proof}

 \section{Characterization of a decomposable data}\label{decomp}
For the notion that we are going to study now, we will include the  trivial partitions
  (all components equal  1) in  admissible data. Notice that they do not modify the total defect.
   Let $u,w,s \in \mathbb{N}$, $U=[u_{1},\dots,u_{s}]$ be a partition of
$u$ and $\mathscr{W}=\{W_{1},\dots,W_{s}\}$ a collection of partitions of $w$.
We define a \emph{product
 partition $U.\mathscr{W}$} as the partition of 
 $uw$ obtained by multiplying each component of $W_{i}$ by 
$u_{i}$ and  taking the union over all  $i=1,\dots,s$. 

  \begin{exe}
  For $u=w=3$, the partition
 $[2,2,2,1,1,1]$ is expressed as  a product partition of $9$ in the following ways: we can express it either as $\bigl[\mathit{1}[2,1],\mathit{1}[2,1],\mathit{1}[2,1]\bigr]$ where $U=[\mathit{1,1,1}]$ and $\mathscr{W}=\{[2,1],[2,1],[2,1]\}$ or $\bigl[\mathit{2}[1,1,1],\mathit{1}[1,1,1]\bigr]$ where $U=[\mathit{2,1}]$ and $\mathscr{W}=\{[1,1,1],[1,1,1]\}$.
  \end{exe}

\begin{defi}  
Let $\mathscr{U}=\{U_{1},\dots,U_{t}\}$ be a family of partitions  of $u$  where the partition  $U_{i}$ contains $s_i$ elements,
 $\mathscr{W}=\bigcup_{i=1}^{t}\mathscr{W}_{i}$ a union of  collections of partitions of $w$ where each $\mathscr{W}_{i}$ is a collection of $s_i$ partitions. Then we define  
$\mathscr{U}.\mathscr{W}$ to be the collection of $t$ partitions of $uw$ where the $i$-$th$ partition is given by the product $U_{i}.\mathscr{W}_{i}, 1\leq i\leq t$.
  Given a  collection $\mathscr{D}$ of partitions   of $d$  and a  non-trivial factorization of $d$, say $d=uw$,
 if there exist $t \in \mathbb{N}$, $\mathscr{U}=\{U_{1},\dots,U_{t}\}$ 
and $\mathscr{W}=\bigcup_{i=1}^{t}\mathscr{W}_{i}$  such that  $\mathscr{D} =\mathscr{U}.\mathscr{W}=\{U_{i}.\mathscr{W}_{i}\}_{i=1}^{t}$ then  
we say that $\mathscr{U}=\{U_{1},\dots,U_{t}\}$ 
and $\mathscr{W}=\bigcup_{i=1}^{t}\mathscr{W}_{i}$  define an algebraic  decomposition $($or factorization$)$ of $\mathscr{D}$. 
\end{defi}

 \begin{exe}
Let $d \in \mathbb{N}$ be a non-prime odd integer. Then every non-trivial
factorization  of $d$ as the product of two positive integers, 
$d=uw$, defines 
a factorization of the admissible data $\mathscr{D}=\{[d]\}$, with admissible non-trivial factors
$\mathscr{U}=\{[u]\}$ and $\mathscr{W}=\{[w]\}$.
\end{exe} 
\begin{prop}\label{fator}
Let $d=uw$ and $\mathscr{D}, \mathscr{U}$, $\mathscr{W}$ 
be   collections of partitions of  $d,u$ and $w$ respectively such that 
$\mathscr{D}=\mathscr{U}.\mathscr{W}$. Then
$\nu(\mathscr{D})=\nu(\mathscr{W})+w \nu(\mathscr{U})$. 
\end{prop}
\begin{proof}
If $\mathscr{D}=\mathscr{U}.\mathscr{W}$, there exist positive integers $t$
and $s_{i}$, for $i=1,\dots,t$, such that 
$\mathscr{U}=\{U_{1},\dots,U_{t}\}$ and
$\mathscr{W}=\bigcup_{i=1}^{t}\mathscr{W}_{i}$, where $U_{i}$ is a partition of $u$ with  $s_{i}$
components and $\mathscr{W}_{i}$ is a
collection of $s_{i}$ partitions of $w$, for $i=1,\dots,t$.
Then $\mathscr{D}=\{U_{i}.\mathscr{W}_{i}\}_{i=1}^{t}$ and
$\nu(\mathscr{D})=\sum_{i=1}^{t}\nu(U_{i}.\mathscr{W}_{i})=\sum_{i=1}^{t}(uw+\nu(\mathscr{W}_{i})-ws_{i})=\sum_{i=1}^{t}\nu(\mathscr{W}_{i})+\sum_{i=1}^{t}w(u-s_{i})=\nu(\mathscr{W})+w
\sum_{i=1}^{t}\nu(U_{i})=\nu(\mathscr{W})+w \nu(\mathscr{U})$.
\end{proof}
\begin{cor}
The factorization  of  admissible data  does not imply admissible factors. \qed
\end{cor}

Recall that a collection of partitions of $d$  is 
called {\it decomposable on $N$} if it is realized on $N$ by a decomposable primitive $d$-fold branched covering.

\begin{prop}\label{nt1}
Let $N$ be a connected closed surface with $\chi(N) \leq 0$.
Admissible data $\mathscr{D}$ are decomposable on $N$   if
 and only if there exists a factorization of
 $\mathscr{D}$ such that its first factor 
is  non-trivial admissible data.
\end{prop}
\begin{proof}
Suppose that $(M,\phi,N,B_{\phi},d)$ is a decomposable primitive branched
covering realizing 
$\mathscr{D}$. Then there exist a surface $K$ and coverings 
 $\psi, \eta$ of degrees $w,u$ respectively such that $\phi=\eta \psi$. Hence 
 $d=uw$ and, since $\phi$ is primitive, there is a
 non-empty subset  $B_{\eta}\subset B_{\phi}$ such that 
$(K,\eta,N,B_{\eta},u)$ is a primitive branched covering with branch data
$\mathscr{\widetilde{U}}$. Note that each
 $x \in B_{\phi}$ determines a partition  of  $u$ (that will be trivial if
$x \in B_{\phi}-B_{\eta}$) and each point in  $\eta^{-1}(x)$ determines a
partition of $w$ (that will be trivial if such a point is not a branch point of
$\psi$). In other words,  $x \in B_{\phi}$ determines a partition 
$U_{x}$ of $u$ and a collection 
$\mathscr{W}_{x}$ of partitions of $w$, such that
$U_{x}.\mathscr{W}_{x}$ is the partition of  $d$ 
that $x$ determines for $\phi$. 
Then $\mathscr{D}=\{U_{x}.\mathscr{W}_{x}\}_{x \in B_{\phi}}$ is
a factorization with an admissible non-trivial first factor $\mathscr{U}=\{U_{x}\}_{x \in B_{\phi}}$, because $\nu(\mathscr{U})=\nu(\mathscr{\widetilde{U}})$
(the differences between $\mathscr{U}$ and $\widetilde{\mathscr{U}}$ are
just the trivial partitions) and
$B_{\eta} \neq \emptyset$. 

Conversely, suppose $d=uw$ and let $\mathscr{D}=\{U_{x}.\mathscr{W}_{x}\}_{x \in B}$ be a factorization of  admissible data, whose first 
factor $\mathscr{U}$ is non-trivial and admissible,  where $B\subset N$ is a finite subset.
By Theorem \ref{D} there exists a primitive branched covering
 $(K,\eta,N,B_{\eta},u)$ realizing 
$\mathscr{U}$, in particular $B_{\eta} \subset B$  and $U_x$ is a trivial partition for  each $x\in B\backslash B_\eta$. By
 Proposition \ref{fator} the second factor $\mathscr{W}$, possibly trivial,
 is admissible. If it is 
 non-trivial, by Theorem \ref{D} there exists 
a primitive branched covering  
 $(M,\psi,K,B_{\psi},w)$ realizing it
  as branch data. Without loss of generality 
we can assume that  $B_{\psi}\subset \eta^{-1}(B)$ and, for each $x\in \eta(B_{\psi})$, $U_x=[u_{x,1},\dots,u_{x,s_x}], $
$\eta^{-1}(x)=\{y_{x,1}, \dots,y_{x,s_x}\}$, $\eta$ has local degree $u_{x,j}$ at $y_{x,j}$, where $\psi$ and the point $y_{x,j}$ determine     
   the partition  $\mathscr{W}_{x,j}$ of  $w$, $1\leq  j \leq s_x$.  Thus 
 $B_{\psi}\subset \eta^{-1}(A \cup B_{\eta})$. Thus 
$(M,\eta \psi, N, B_{\eta}  \cup \eta(B_{\psi}),$ $d=uw)$ is a decomposable primitive branched covering with branch data
 $\mathscr{D}$. 
If $\mathscr{W}$ is trivial the result follows from Proposition \ref{indice}.
\end{proof}
\begin{cor}\label{comp}
If $\mathscr{D}=\{D_{1},\dots,D_{t}\}$ is decomposable on $N$ then 
each component of $D_{i}$, for $i=1,\dots,t$, is a product 
of two integers (one of them or both can equal 1) less
than or equal to the degrees of the two coverings in the
decomposition, respectively.  \qed
\end{cor}

\begin{exe}
A primitive branched covering like $(M,\phi,T_{1},\{x\},4)$ is
indecomposable: notice that possible  admissible data are
 $[1,3]$ and $[2,2]$. But $[1,3]$ is not realized 
by a decomposable primitive  branched covering on $T_{1}$ by
 Corollary \ref{comp}, and every factorization of
 $[2,2]$  has either a trivial or a non-admissible first factor on $T_{1}$. 
 \end{exe}
\section{ Realization of   branch  data by  indecomposable  branched coverings}
We begin this section by giving an example of  admissible data which  admits a decomposable and an
{\it in}decomposable realizations on  $T_{1}$, at 
the same time.
\begin{exe}
Let us consider the realizable data $\mathscr{D}=\{[3,2,2,2],[3,2,2,2]\}=\{[\mathit{1}[3],\mathit{2}[1,1,1]],[\mathit{1}[3],\mathit{2}[1,1,1]]\}$ on  
$T_{1}$,  with the non-trivial admissible first factor
$\mathscr{U}=\{\mathit{[1,2],[1,2]}\}$ and the 
second  factor
$\mathscr{W}=\cup_{i=1}^{2}\{[3],[1,1,1]\}_{i}$. By Proposition
\ref{nt1}, $\mathscr{D}$ is  decomposable on $T_{1}$ but, 
on the other hand, we can define the following representation:
\begin{small}
\begin{eqnarray*}
\rho:\pi_{1}(T_{1}-\{x,y\})= \langle
 a,b,\mathbf{u_{x}},\mathbf{u_{y}}| [a,b]\mathbf{u_{x}u_{y}}=1 \rangle & \lra & \Sigma_{9}, \\
 a & \longmapsto & (1\;4\;5\;6\;7\;8\;9\;3\;2), \\
 b & \longmapsto & (2\;4\;5\;6\;7\;8\;9\;3), \\
\mathbf{u_{x}} & \longmapsto & (1\;2\;3)(4\;5)(6\;7)(8\;9), \\
\mathbf{u_{y}} & \longmapsto & (1\;2\;3)(4\;5)(6\;7)(8\;9)
\end{eqnarray*}
\end{small}
where $G=\Im \rho$  is a transitive primitive permutation group because it
contains a $9$- and an $8$-cycles $($see Example \ref{e1}$)$. Then, by Proposition  \ref{yo},
the  branched covering that it determines is indecomposable $($and, hence, primitive$)$.
\end{exe}
The goal of this section is to show that every admissible data are
realizable by an indecomposable (and, hence, primitive) branched
covering on a connected
closed surface $N$, with $\chi(N)\leq 0$.
The following theorem (that will be proved in the next section) solves the
case when $N$ is either the torus $T_{1}$ 
or the Klein bottle $P_{2}$ with only one branch  point, and it will be
used  to solve the general case. 
\begin{thm}\label{nt2} 
Let $d \in \mathbb{N}$ be a  non-prime and $D=[d_{1},\dots,d_{t}]$ a non-trivial partition
of $d$ such that $\nu(D) \equiv 0 \pmod{2}$. Then there exist 
indecomposable  $($and, hence, primitive$)$ coverings on the torus $T_{1}$ and on the
Klein 
bottle  $P_{2}$ respectively, realizing  $D$ as branch data.
\end{thm}

\begin{thm}\label{nt3}  Every non-trivial admissible data   are  realizable  by an indecomposable
 $($and, hence, primitive$)$ branched  covering  on any $N$  with $\chi(N)\leq 0$. 
\end{thm}
\begin{proof}
Let $\mathscr{D}=\{D_{1},\dots,D_{r}\}$ be admissible. For 
 $i=1,\dots,r$, choose $\gamma_{i} \in \Sigma_{d}$ such that its cyclic
 structure is  $D_{i}$.

 If   $\prod_{i=1}^{t}\gamma_{i} \neq 1_{d}$, the cyclic structure of this product
 is a  non-trivial partition  $D=[d_{1},\dots,d_{t}]$ of $d$ with
$t<d$ and $\nu(D) \equiv \nu(\mathscr{D})\equiv 0 \pmod{2}$. Thus, by Theorem
\ref{nt2} it is realizable by an indecomposable (and, hence, primitive)  branched 
covering on $T_{1}$ (on $P_{2}$, respectively). By Proposition \ref{yo}, there exist permutations $\lambda, \beta \in \Sigma_{d}$
($\omega, \theta$, respectively) such that the cyclic structure of 
$[\lambda,\beta]$ ($\omega
\theta \omega \theta^{-1}$ or $\omega^{2}\theta^{2}$, respectively) is   $D$ and the permutation group
 $G_{1}:=\langle
\lambda, \beta \rangle$ ($G_{2}:=\langle\omega,\theta \rangle$, respectively)   is transitive and  primitive.
 Thus we define the representation:
 $\rho_{1}: \langle
\{ a_{i},b_{i}\}_{i=1}^{g}\cup\{ \mathbf{u_{j}}\}_{j=1}^{t}| \Pi_{j=1}^{t} \mathbf{u_{j}}=\Pi_{j=0}^{g-1}[b_{g-j},a_{g-j}] \rangle  \lra  \Sigma_{d}$ sending
 $a_{1}  \longmapsto  \beta$,
 $b_{1} \longmapsto \lambda$,
$\mathbf{u_{i}} \longmapsto  \gamma_{i}$ and
$\{a_{j},b_{j}\}_{j=2}^{g} \longmapsto 1_{d}$,
for $N=S_{g}$ (respectively for $N=P_{g}, g>1$,  $\rho_{2}: \langle \{ a_{i}\}_{i=1}^{g}\cup\{ \mathbf{u_{j}}\}_{j=1}^{t}| \Pi_{j=1}^{t} \mathbf{u_{j}}=\Pi_{j=0}^{g-1}a_{g-j}^{2} \rangle  \lra  \Sigma_{d}$, sending  $a_{1} \longmapsto \theta$, $a_{2} \longmapsto \omega$, 
$\mathbf{u_{i}} \longmapsto \gamma_{i}$ and $\{a_{j}\}_{j=3}^{g}\longmapsto 1_{d}$).
Since $G_{1}<\Im(\rho_{1})$, 
$\Im{\rho_{1}}$ is transitive and primitive. Thus, by
 Proposition \ref{yo}, the primitive branched covering that it determines
 is indecomposable (analogously for $G_{2}$).

 If $\prod_{i=1}^{t}\gamma_{i}=1_{d}$ and there is some 
 $\gamma_{i}$ with a cycle of length $\geq 3$, we change  $\gamma_{i}$ 
by $\gamma_{i}^{-1}$. If $d>2$ and each $\gamma_{i}$ is a product of cycles with length
 $\leq 2$, we change a symbol in a cycle of length 2 by a symbol in another
 cycle. Thus, we do not change the cyclic structure of the $\gamma_{i}$'s, 
 the  new product $\prod_{i=1}^{t}\gamma_{i}$ is different from $1_{d}$ and we are
 in the case before. If $d=2$ then $\mathscr{D}$ is obviously realizable on any $N$; the corresponding branched 
covering is indecomposable,  since $2$ is a prime.
\end{proof}
\section{Proof of Theorem \ref{nt2}}
Let $d \in \mathbb{N}$ be not a prime and $D=[d_{1},\dots,d_{t}]$ a non-trivial partition
of $d$ such that $\nu(D)=d-t \equiv 0 \pmod{2}$. 
By  Proposition \ref{yo}, it is enough to prove the existence
 of permutations
 $\lambda, \beta,\omega, \theta \in \Sigma_{d}$ such that $[\lambda, \beta]$ for $T_{1}$ and, either  $\omega \theta 
 \omega \theta^{-1}$ or $\omega^{2}\theta^{2}$ for $P_{2}$, have cyclic structure 
 $D$ and the
 permutation groups 
  $G_{1}:=\langle \lambda, \beta \rangle$  and $G_{2}:=\langle \omega, \theta \rangle$ are primitive, respectively. We divide the proof 
in according to the following three cases:
\begin{itemize}
\item[(1)] $t=1$,
\item[(2)] $D=[2, \dots,2]$,
\item[(3)] $t>1$ and $d_{i} \neq 2$ for some $i \in \{1,\dots,t\}$.
\end{itemize}
 For (1), the permutations defined in \cite{BGKZ2} (proof of
 theorem  2.2 and theorem 2.3, case $r=1$) work (see below). For 
(2) and (3), we use the following idea: given $\alpha \in D$,  we define a permutation
$\beta \in \Sigma_{d}$ such that $\beta$ and $\alpha \beta$ have the same
cyclic  structure, moreover the permutation group $H:=\langle \alpha, \beta\rangle<\Sigma_d$ is transitive and primitive. Then there exists $\lambda \in \Sigma_{d}$ such that
$\alpha=[\lambda, \beta] $ and $\langle
\lambda, \beta \rangle$ is also transitive and  primitive (since it conains $H$). Analogously, 
since $\alpha \beta$ and
$\beta^{-1}$ are conjugate, there exists $\omega \in \Sigma_{d}$ such that
$\alpha \beta = \omega \beta^{-1} \omega^{-1}$. We define
$\theta=\omega^{-1}\beta^{-1}$, then $\alpha=\omega^{2}\theta^{2}$ and 
 $\langle \omega, \theta \rangle$ is transitive and  primitive.
\subsection*{Case (1)}
 If $t=1$ then $D=[d]$ and $d=2k+1$ with $k>0$. For $T_{1}$ we define
    permutations 
    $\lambda=(k+1\quad k+2\;\dots\;2k \quad 2k+1)$  and
\begin{displaymath}
\mathbf{\beta=} 
\left( \begin{array}{ccccccc} 
1 &\dots & k & k+1& k+2& \dots  & 2k+1\\
2k+1 & \dots &  k+2 & k+1 & 1 & \dots & k
\end{array} \right),
\end{displaymath}
then $[\lambda,\beta]=(1 \; 2 \dots k \; k+1\;k+2 \dots 2k+1)$ has cyclic
 structure  $D$ and $G_{1}=\langle 
 \lambda, \beta \rangle$ is transitive. 
Due to Example \ref{REF} with $d=2k+1$ and $\ell=k+1$, $G_1$ is a primitive 
permutation group. For
 $P_{2}$, we
 define  $\omega=\lambda$,
\begin{displaymath}
 \mathbf{\theta=} 
\left( \begin{array}{ccccccc}
1 & 2&\dots  & k+1& k+2& \dots & 2k+1\\ 
2k+1 &k+1& \dots & 2k & 1 & \dots & k
\end{array} \right),
\end{displaymath}   
note that $\omega \theta \omega \theta^{-1}=(1\;2\dots2k+1)$ has cyclic
structure $D$ and, analogously to the orientable case, we conclude
that 
$G_{2}=\langle 
\omega, \theta \rangle$ is a primitive permutation group.
\subsection*{Case (2)}
If $D=[2,\dots,2]$ then $d=2t$ with $t$ even. 
 Let  $\alpha=(1\;2)(3\;4)\dots(t-1\;t) \dots (2t-1\;2t) \in D$
 and 
$\beta=(1\;3\dots t-1 \dots 2t-1 \; 
4)(2 \; 6 \dots t \dots 2t)$, i.e. the first cycle is defined by the increasing  sequence of odd numbers, from $1$ to $2t-1$, followed by the even number 4, and the second cycle is defined by the increasing sequence of even numbers from $2$ to $2t$ without $4$. By induction on  $t$ we can see that  $\beta$
and $\alpha \beta$ are conjugate, then there exists $\lambda \in \Sigma_{d}$
 such that
$\alpha 
\beta=\lambda \beta \lambda^{-1}$. Thus $[\lambda,\beta]=\alpha$ has cyclic
 structure $D$. Consider the transitive group 
 $G_{1}:=\langle \lambda, 
\beta\rangle$. 
Let $\| 1 \| $ denote a non-trivial block of $G_1$ containing the element $1$. If   $\| 1 \| \subset 1^{\langle \beta \rangle}$ is 
contained in the orbit of $1$ by $\langle \beta \rangle$, its cardinality  $\#\| 1 \|$ is a common factor of
$t+1$ and
$2t$, thus $\#\| 1 \|=1$. On the other hand, if $\| 1 \|$
 contains elements of both cycles of  $\beta$ then $\gcd( t+1,t-1)\neq 1$, which is impossible. Since the blocks are trivial, $G_{1}$ is primitive and we resolved for
 $T_{1}$. For $P_{2}$, we use the idea of the sketch of the proof in the beginning of this section to define  $\omega,\; \theta \in
 \Sigma_{d}$ such that
$\alpha=\omega^{2}\theta^{2}$. Since 
 $G_{2}:=\langle \omega, \theta \rangle= \langle
 \omega,\beta 
\rangle$ is transitive and $\beta$ determines the primitivity of  $G_{1}$, then $G_{2}$ is also
 primitive.  
 \subsection*{Case (3)}
 Finally suppose $t>1$ and $d_{i} \neq 2$ for some  $i
\in \{1,\dots,t\}$. We define $\delta_{0}:=0$,  
$\delta_{i}:=\sum_{j=1}^{i}d_{j}$, 
 $C_{i}:=(\delta_{i-1}+1 \dots \delta_{i}) \in \Sigma_{d}$ and  the
sequence 
$\Delta_{i}:=\{\delta_{i-1}+k\}_{k=2}^{d_{i}}$, for $i=1,\dots,t$. Since $t<d$, we impose
$d_{1}>1$ thus $\Delta_{1}\neq \emptyset$. Without loss of generality
let
\begin{eqnarray*} 
\alpha:=\prod_{i=1}^{t}C_{i}=(1\;\underbrace{2\dots \delta_{1}}_{\Delta_{1}})(\delta_{1}+1\;\underbrace{\delta_{1}+2\dots \delta_{2}}_{\Delta_{2}})\dots(\delta_{t-1}+1         
\; \underbrace{\delta_{t-1}+2\dots\delta_{t}}_{ \Delta_{t}}) \in D. 
\end{eqnarray*}
 Since there is $d_{i}\neq 2$ then either $Fix(\alpha) 
\neq \emptyset$ or $\alpha^{2}\neq 1_{d}$. Define the $d$-cycle
\begin{eqnarray*}
 \beta:=(1\quad \delta_{1}+1\quad \delta_{2}+1 \dots
\delta_{t-1}+1\quad \Delta_{1}\quad \Delta_{2} \dots \Delta_{t}).
\end{eqnarray*}
Denote by
$E_{i},O_{i}$ the increasing sequences of even and odd elements 
in $\overline{Supp(C_{i})}:=\{\delta_{i-1}+1\}\cup\Delta_i$, 
respectively. Note that $\alpha \beta$ is a $d$-cycle obtained by  concatenating  these sequences, thus:
\begin{displaymath} \label{gf1}
\begin{small}
\alpha \beta= \left\{ \begin{array}{ll}
(O_{1}E_{2}O_{3} \dots 
E_{t}E_{1}O_{2}E_{3}\dots O_{t}) & \textrm{if $t$ is even,}\\
(O_{1}E_{2}O_{3} \dots 
O_{t}E_{1}O_{2}E_{3}\dots E_{t})& \textrm{if $t$ is odd}. \\
\end{array} \right.
\end{small}
\end{displaymath}
Here we use that  the $t$-th term is either $E_t \ni d$ if $t$ is even or $O_t\ni d$
if $t$ is odd, hence it is non-empty and terminates by $d$, hence  the next term is $E_1$.
Hence $\alpha \beta$ and 
$\beta$ are conjugate and there is
$\lambda \in \Sigma_{d}$ such that $\alpha \beta=\lambda
\beta \lambda^{-1}$ and $\alpha=[\lambda,\beta]$.

Let  $H:=\langle \alpha,\beta \rangle$, it is obviously transitive. We assert that $H$ is a primitive
permutation group (at  least after a suitable permutation of $d_1,\dots , d_t$, see below). By contradiction, 
suppose the existence of a non-trivial
divisor   $n$ of  $d$ such that $H$ determines $n$ blocks of cardinality 
 $d/n$. We have the following consequences:
 \begin{enumerate}
 \item[I)] 
$\| 1 \|= \| 1^{\beta^{n}}\|$  and 
$\| i \|= \{i,i^{\beta^n},i^{\beta^{2n}},\dots,i^{\beta^{(d/n-1)n}}\}=\{i^{\beta^{kn}} | k\in \mathbb{Z}\}$
for any $i\in \{1, \dots,d\}$. Hence 
$R:=\{1,1^{\beta},1^{\beta^{2}},\dots,
1^{\beta^{n-1}}\}$  is a set of representatives of all blocks, and consecutive elements in $\beta$ are in different blocks.
\begin{proof}
 Since $\beta$ is a $d$-cycle, the blocks are completely determined by
 the cycles of  $\beta^{n}$. 
  \end{proof} 
  \item[II)]
  $n<t$ (at least if $d_1=max\{d_1,\dots,d_t\}$ and $d_2=min\{d_1,\dots,d_t\})$.  Hence $R=\{1,\delta_{1}+1,\delta_{2}+1,\dots,\delta_{n-1}+1\}$, $1^{\beta^{n}}=\delta_{n}+1$ and 
$\Gamma:=\{\| 1 \|,\| \delta_{1}+1 \|,\dots,\|
\delta_{n-1}+1\| \}$ is the system of all blocks.
\begin{proof}
 If $n \geq t$, 
there is $i \in 
\{1,\dots,t\}$ such that $1^{\beta^{n}} \in \Delta_{i}$,
\begin{eqnarray*}
 \beta=(1\;\underbrace{\delta_{1}+1\dots \delta_{i}+1 \dots
\delta_{t-1}+1\quad \Delta_{1}\quad \Delta_{2} \dots\hspace*{0.6em}}_{n} \hspace*{-0.6em}\Delta_{i}  \dots \Delta_{t}).
\end{eqnarray*}
If $i>1$   then 
$1^{\alpha^{-1}}$ and $1^{\beta^{n}\alpha^{-1}}$, both in $R$,
are  in the same block by the first assertion in I), a contradiction with the second assertion in I).  If
$i=1$ (and hence $1^{\beta^n}=n-t+2$) and there is $d_j>2$ then necessarily $d_j\leq n-t+2=1^{\beta^n}\in \Delta_1=\{2,\dots,d_1\}$,
provided that $d_1=max\{d_1,\dots,d_t\}$. In fact, otherwise  we put $j=1$ and by applying $\alpha$ to $1$  and  
$1^{\beta^{n}}$  we obtain the
elements $1^{\alpha} =2=1^{\beta^{t}}$ and 
$1^{\beta^n\alpha}=n-t+3=1^ {\beta^{n+1}}$ in
 $\Delta_{1}$, 
in the same block, and by I) this implies $n|(t-1)$, hence $n<t$ or $t=1$, a contradiction.
 Then $1^{\beta^{n}}=n-t+2=d_1$, hence $1^{\alpha}=2<d_1=1^{\beta^n}$ and $1^{\beta^n\alpha}=d_1^{\alpha}=1$, both in $R$, are different elements 
in the same block,
 \begin{eqnarray*}
 \beta=(1\;\delta_{1}+1\dots 
\delta_{t-1}+1\quad \overbrace{\underbrace{2\dots n-t+2}_{n-t+1}\dots d_{1}}^{\Delta_{1}}\quad \Delta_{2} \dots\Delta_{i} \dots \Delta_{t}),
\end{eqnarray*}
a contradiction with I). If $i=1$ and $d_{j}\leq 2$ for $j=1,\dots,t$, then $\Delta_{j}$
has at most one element and
$\alpha^{2}=1$. Then  Fix$(\alpha)\neq
\emptyset$ and $\alpha$ is the product of an even number of transpositions, because $\nu(D) \equiv 0 \pmod{2}$. We put  $C_{1}=(1\;2)$, $C_{2}=(3)$ and
$C_{3}=(4\; 5)$, then  $\beta=(1\;3\;4\dots\delta_{t-1}+1 \;\; 2\; 5\dots)$. Now $i=1$ implies $n-t+1\leq|\Delta_1|=1$, hence
 $n=t$,  then   
$\|1\|= \|2 \|$. Hence  $\| 1^{\beta \alpha} \| = \|2^{\beta \alpha}\|$, i.e. we obtain 
$3$ and $4$, consecutive in $\beta$, in the same block,  
 a contradiction with I).
\end{proof}
\item[III)]
 $n \nmid t$ (at least if $d_1=max\{d_1,\dots,d_t\}$ and $d_2=min\{d_1,\dots,d_t\})$. Hence $n\neq 2$ and $1^{\alpha} \notin \| 1 \|$.
 \begin{proof}
  If $n \mid t$ then I) and   definition of  $\beta$ imply  $1^{\alpha} \in
 \| 1 \|$ (since $1^{\alpha}=2=1^{\beta^t}$if $d_1\ne 1)$,  then: if there is some $d_{i}>2$, we put $i=1$ and 
by  applying $\alpha$ to $1$ and $1^{\alpha}=2$,  
we obtain the consecutive elements 2 and 3 in $\Delta_1$ in the same block, a contradiction with I).
If $n\mid t$ and $d_i\leq 2$ for $i=1,\dots,t$,  then from $\| 1 \|=\| 2 \|$ we obtain a contradiction as in the last part of II).
This proves that $n\nmid t$. If $n=2$  then $n\mid d$ and $n\nmid t $ contradict to  Hurwitz's condition $d\equiv t \ (mod \ 2)$. Now 
$1^{\alpha}=2=1^{\beta^t} \notin \{1^{\beta^{kn}} |k\in \mathbb{Z}\}=\| 1 \|$,   since $n\nmid t$.
\end{proof}
\item[IV)] $\alpha=(1\dots\delta_{1})(\delta_{1}+1 \dots \delta_{2})\dots (\delta_{t-1}+1\dots \delta_{t})$ and $\beta=(1\;\delta_{1}+1\dots\delta_{t-1}+1\;\Delta_{1}\dots\Delta_{t})$ induce permutations  $\bar {\alpha}, 
\bar {\beta}$ of $\Gamma=\{\| 1 \|,\| \delta_{1}+1 \|,\dots,\|
\delta_{n-1}+1\| \}$. 
 By definition of $\beta$ we have 
\begin{eqnarray*}
\bar {\beta}=(\| 1 \|\;\| \delta_{1}+1 \|\;\|
\delta_{2}+1 \| \dots
\| \delta_{n-1}+1 \|).
\end{eqnarray*}
In order to determine  $\bar {\alpha}$, it is enough to know in which blocks are the elements of $\overline {Supp(C_{i})}$, for $i=1,\dots,t$. For $C_{1}$ note that:
\item[V)]
Every element of $\| 1 \|$ is in a cycle of
 $\alpha$ with length bigger than 2 (provided that the same assumptions as in  II) and III) hold). In
particular $d_{1}>2$.
\begin{proof}
 If $x \in
\| 1 \| \cap Fix(\alpha)$ then $1^{\alpha}\in \| 1 \|$, a contradiction with III). If $x \in \| 1\|$ is in a transposition of $\alpha$ then each
element of $\| 1 \|$ is also in a transposition (see I)) and $d_{i}\leq 2$ for all
 $i$ (otherwise we put $d_{1}>2$, then $1\in \| 1\|$ is in a $d_1$-cycle which is not a transposition).  
 Put
 $d_{1}=2$ and $d_{2}=1$. By I) and II),
   $\| 1^{\alpha} \| = \|(\delta_{n}+1)^{\alpha}\|$, but this is impossible because for all $i$, $\Delta_{i}$ has at most one element, in 
particular $\Delta_{1}=\{1^{\alpha}\}$, $\Delta_{2}=\emptyset$ and $\Delta_{n+1}=\{(\delta_n+1)^{\alpha}\}$ (respectively $\Delta_{n+1}=\emptyset$),  then 
the power of
 $\beta$ that takes $1^{\alpha}$ to $(\delta_{n}+1)^{\alpha}$ is
 smaller than $n$ (respectively  equals  $n-t$, which is not divisible by $n$ in according to III)) and  they are represented by different elements in $R$ (see I)),  a contradiction.
\end{proof}
\item[VI)]  If $\| \delta_{j-1}+1 \| \cap Supp(C_{1}) \neq \emptyset$ then $d_{j}>2$, $2 \leq j\leq t$.
\begin{proof} Suppose that  $x \in Supp(C_{1}) \cap \| \delta_{j-1}+1 \|$ and
 $d_{j}\leq 2$. Since $d_{1}>2$ by V), by applying successively 
$\alpha^2$ to  $x$ (respectively to $\delta_{j-1}+1$),
we obtain elements in $\Delta_{1}$ (respectively  $\delta_{j-1}+1$ itself) in 
the same block such that the power of $\beta$ that takes one to the other is $\leq 2$, a contradiction with III) and I). 
\end{proof}
\item[VII)] 
$d_{1}<n$ and $\| 1\|\cap \Delta_1=\emptyset$. Hence 
different elements in $Supp(C_{1})$ determine different blocks and $(\| 1 \|\; \| 2\| \dots \| d_{1}\|)$ is
a $d_1$-cycle
 of $\bar{\alpha}$. 
\begin{proof}
If $d_{1} > n$ (or at least $\| 1\|\cap \Delta_1 \ne \emptyset$), then $n|d_1$ and $\delta _{t-1}+1\in \| 1\|$, thus the  first 
element of $\Delta_{1}$ in $\| 1\|$ is $n+1$, 
since otherwise, by applying $\alpha$ to this element and to  $1$, we obtain in  $\Delta_{1}$ elements in the
same block that determine different elements of $\Gamma$.  If $d_1=n$ then  similar arguments show that $\delta_{t-1}+1\in \| 1\|$.
Then

\begin{eqnarray*}
\beta=(1\; \delta_{1}+1\dots \delta_{n-1}+1\dots \delta_{t-1}+1\;\overbrace{2 \dots n \dots d_{1}}^{\Delta_{1}}\; \overbrace{\delta_{1}+2 \dots \delta_{2}}^{\Delta_{2}} \dots \Delta_{t})
\end{eqnarray*}
 and, by I) and definition of $\bar{\beta}$, necessarily $\delta_{t-1}+1
\in \|1 
\|$, $2 \in \| \delta_{1}+1
\|$ (then $d_{2}>2$ and $\Delta_{2}\neq \emptyset$, by VI)), $n \in \| \delta_{n-1}+1 \|$ and  $\bar{\alpha}=\bar{\beta}=(\| 1 \|\;\| \delta_{1}+1 \|\;\|
\delta_{2}+1 \| \dots
\| \delta_{n-1}+1 \|)$. Moreover $d_{1}\in \| \delta_{n-1}+1 \|$ 
and $\delta_{1}+2 \in
 \| 1 \|$ (since $d_2>1$), but
$(\delta_{1}+1)^{\alpha}=\delta_{1}+2$ then $\| \delta_{1}+1
\| ^{\bar{\alpha}}=\| \delta_{1}+2 \|=\| 1 \|$ and $n=2$, a contradiction with III).
\end{proof}
\item[VIII)] The cycle in VII) can be represented as $(\| 1 \| \;\| \delta_{1}+1 \| \; \| \delta_{2}+1\|
  \dots 
  \| \delta_{d_{1}-1}+1 \| )$ which implies $d_{1}=2$, a contradiction with V).
  \begin{proof}
 Let  $t=nq+r$ with $q, r \in \mathbb{Z}^{+}$, $r<n$ (then $q\geq 1$ by II) and $r\geq 1$ by III)). Definition of  $\beta$
with VII)  imply $n>r+d_{1}-2$, 
\begin{small}
\begin{eqnarray*}
 \beta=(1\quad \delta_{1}+1 \dots \delta_{n}+1 \dots \delta_{nq}+1\underbrace{ \overbrace{ \dots
\delta_{t-1}+1\; \underbrace{2\dots d_{1}}_{\Delta_{1}}}^{r+d_{1}-2}\;\underbrace{\delta_{1}+2 \dots \delta_{2}}_{\Delta_{2}} \dots \hspace*{0.6em}}_{n}\hspace*{-0.6em}\Delta_{i} \dots \Delta_{t})
\end{eqnarray*}
\end{small}

\noindent  and   $(\delta_{nq}+1)^{\beta^{n}}\in \|1\|$ by I),   with either 
  $(\delta_{nq}+1)^{\beta^{n}}=1 $ (if $d=nq+n,$ i.e. $t>d-n$) or 
$(\delta_{nq}+1)^{\beta^{n}} \in \Delta_{i} $ for some $i$ in
  $\{2,\dots,t\}$ (if $d>nq+n,$ i.e. $t\leq d-n$). Suppose that   $d>nq+n,$ thus $(\delta_{nq}+1)^{\beta^{n}} \in \Delta_{i}.$    By V), 
$\Delta_{i}$ has more than 1 element and, since
  $1^{\alpha^{-1}}$ and $(\delta_{nq}+1)^{\beta^{n}\alpha^{-1}}$
 are in the same block, necessarily
 $(\delta_{nq}+1)^{\beta^{n}}=\delta_{i-1}+2$ 
  (the first element of $\Delta_{i}$). Then $1^{\alpha \beta^{-1}}$ and
  $(\delta_{nq}+1)^{\beta^{n}\alpha \beta^{-1}}$ are in $\| 1
  \|^{\bar{\alpha} \bar{\beta}^{-1}}$, but $(\delta_{nq}+1)^{\beta^{n}\alpha \beta^{-1}}=\delta_{i-1}+2$. 
Hence  $\| 1 \|^{\bar {\alpha}\bar {\beta}^{-1}}=\| 1 \|$, $2 \in
  \| \delta_{1}+1 \|$
  and, by definition of $\alpha$ and  $\beta$, 
  $(\| 1 \| \;\| \delta_{1}+1 \| \; \| \delta_{2}+1\|
  \dots 
  \| \delta_{d_{1}-1}+1 \| )$ is a 
  cycle of $\bar {\alpha}$.
     Then $2 \in \| \delta_{1}+1 \|$ (then
  $d_{2}>2$ by VI)), $d_{1} \in \| \delta_{d_{1}-1}+1 \|$ and $\delta_{1}+2 \in \| \delta_{d_{1}}+1 \|$. But
  $\delta_{1}+2=(\delta_{1}+1)^{\alpha}$, then 
$\| \delta_{2}+1  \|=\| \delta_{1}+1  \|^{\bar{\alpha}}=\| \delta_{d_{1}}+1\|$, hence $n | (d_1-2)$ by I), hence $d_1=2$ by VII). 
 Suppose that $d=nq+n$, thus $t>d-n$. Since $\delta_{nq}+1\in \| 1\|$, it follows from V) that  $\Delta _{nq+1}\ne \emptyset$ and $(\delta_{nq}+1)^{\alpha}\in \Delta _{nq+1}$, 
thus $1^{\alpha}$ and $(\delta_{nq}+1)^{\alpha}$ are different  elements in $\Delta_1$  
and $\Delta_{nq+1}$,  hence the power  of $\beta$ that takes one to the other is  smaller than $d-t<n$. It follows from I) that these 
two elements are in different blocks, a contradiction. 
  \end{proof}
  \end{enumerate}
 Then  any block of
   $H$ is trivial and $H$ is primitive. Finally, since  $H<G_{1}=
  \langle \lambda, 
  \beta \rangle$ then $G_{1}$ is also primitive and we resolve for $T_{1}$.
 For $P_{2}$ we use the idea of the sketch of the proof at the beginning  to define  $\omega$,
 $\theta$ and  $G_{2}=\langle \omega, \theta \rangle$. We conclude that $G_{2}$ is
  primitive because it contains $H$. $\qed$

\vspace{1cm}
Departamento de Matem\'atica DM-UFSCar

Universidade Federal  de S\~ao Carlos

Rod. Washington Luis, Km. 235. C.P 676 - 13565-905 S\~ao Carlos, SP - Brasil

nbedoya@dm.ufscar.br

and

Departamento de Matem\'atica

Instituto de Matem\'atica e Estat\'istica

Universidade de S\~ao Paulo

Rua do Mat\~ao 1010, CEP 05508-090, S\~ao Paulo, SP, Brasil

dlgoncal@ime.usp.br.


\begin{thebibliography}{7}
\markright{Bibliography}
\bibitem [1]{B} Baildon, John D.:
     {\it Open simple maps and periodic homeomorphisms}. Proc. Amer. Math. Soc.
 {\bf 39}, 433-436 (1973)
 \bibitem [2]{BBZ}
    Bogataya, S. I.; Bogaty{\u\i}, S. A.; Zieschang, H.:
    {\it On compositions of open mappings}.
   Mat. Sb.
  {\bf193}, 3-20 (2002) 

  \bibitem[3]{BGKZ2} Bogaty{\u\i}, S.; Gon\c calves, D. L.; Kudryavtseva, E.; Zieschang, H.: \emph{Realization of Primitive Branched Coverings over Closed
  Surfaces following the Hurwitz approach}.  Cent. Eur. J. Math.  1,  no. 2, 184--197 (electronic) (2003)



\bibitem[4]{BGKZ1}Bogaty{\u\i}, S.; Gon\c calves, D. L.; Kudryavtseva, E,; Zieschang, H.: \emph{Realization of Primitive Branched Coverings over Closed
 Surfaces}.  Advances in topological quantum field theory,  297--316, NATO
  Sci. Ser. II Math. Phys. Chem., {\bf 179}, Kluwer Acad. Publ., Dordrecht (2004)






\bibitem[5]{BGZ}Bogaty{\u\i}, S.; Gon\c calves, D. L.; Zieschang, H.: \emph{The
  minimal number of roots of surface mappings and quadratic equations in free
  groups}. Mathematische Zeitschrift \textbf{236}, 419-452 (2001). 







  \bibitem[6]{BM}
    Borsuk, Karol; Molski, R.:
    {\it On a class of continuous mappings.}
   Fund. Math.
45, 84-98 (1957)
\bibitem[7]{DM} Dixon, J. D.; Mortimer, B.: \emph{Permutation Groups}. Springer-Verlag (1996).
\bibitem[8]{EKS} Edmonds, A. L.; Kulkarni, R. S.; Stong, R. E.:
  \emph{Realizability of Branched Coverings of
  Surfaces}. Trans. Amer. Math. Soc. \textbf{282}, 773-790 (1984)
  \bibitem[9]{Ez} Ezell, Cloyd L.:
     {\it Branch point structure of covering maps onto nonorientable
              surfaces}.
   Trans. Amer. Math. Soc. {\bf 243}, 123-133 (1978)
\bibitem[10] {GKZ}
    Gon{\c{c}}alves, Daciberg L.; Kudryavtseva, Elena; 
              Zieschang, H.:
     {\it Roots of mappings on nonorientable surfaces and equations in
              free groups}. Manuscripta Math. {\bf 107} 3, 311-341 (2002)
\bibitem[11]{GN}
    Guralnick, Robert M. and Neubauer, Michael G.:
    {\it Monodromy groups of branched coverings: the generic case}.
 Recent developments in the inverse {G}alois problem
              ({S}eattle, {WA}, 1993). Contemp. Math., {\bf 186},
     325-352 (1995)
\bibitem[12] {Hurwitz}
    Hurwitz, A.:
    {\it \"Uber {R}iemann'sche {F}l\"achen mit gegebenen
              {V}erzweigungspunkten}.
   Math. Ann. {\bf 39} 1, 1-60 (1891)
\bibitem[13]{Hu} Husemoller, D. H.: \emph{Ramified Coverings of Riemann
    Surfaces}. Duke Math. J. \textbf{29}, 167-174 (1962). 





\bibitem[14]{J} Jacobson,  N.: \emph{ Basic Algebra I.} Second edition W. H. Freeman and Company, N. Y.  (1985)



\bibitem[15] {KJ}
    Krzempek, J.:
    {\it Covering maps that are not compositions of covering maps of
              lesser order}. Proc. Amer. Math. Soc., {\bf 130} 6, 1867-1873 (2002)
\bibitem[16]{LZ}
    Lando, Sergei K. and Zvonkin, Alexander K.:
     {\it Graphs on surfaces and their applications}.
    Encyclopaedia of Mathematical Sciences,
    {\bf 141},
     With an appendix by Don B. Zagier,
              Low-Dimensional Topology, II. 
 Springer-Verlag (2004)





 \bibitem[17]{Muller}
    M{\"u}ller, Peter:
     {\it Primitive monodromy groups of polynomials}.
 Recent developments in the inverse {G}alois problem
              ({S}eattle, {WA}, 1993), Contemp. Math.
    {\bf 186}, 385-401 (1995)
\bibitem[18]{R} Ritt, J. F.: \emph{Prime and composite
    polynomials}. Trans. Amer. Math. Soc. Vol. 23, \textbf{1}, 51-66 (1922)

    \bibitem[19]{S}
    Sieklucki, K.:
    {\it On superpositions of simple mappings}.
   Fund. Math. {\bf 48}, 217-228 (1959/1960)



      \bibitem[20] {Why}
    Whyburn, Gordon Thomas:
     {\it Analytic topology}. American Mathematical Society Colloquium Publications. {\bf XXVIII} (1963)
     \end{thebibliography}
\end{document}